\documentclass[a4paper,10pt]{article}
\usepackage[russian,english]{babel}
\usepackage[cp1251]{inputenc}
\usepackage[left=44mm,right=40mm,top=60mm,bottom=52mm]{geometry}

\usepackage{graphicx}
\usepackage{epsfig}
\usepackage{amsthm}
\usepackage{amsmath}
\usepackage{amsfonts}
\usepackage{amssymb}
\usepackage{amscd}

\newtheorem{theorem}{Theorem}

\newtheorem{definition}{Definition}
\newtheorem{lemma}{Lemma}

\newtheorem{remark}{Remark}
\usepackage{mathrsfs} 

\begin{document}
\title{About ergodicity and polynomial convergence rate of Generalized Markov modulated Poisson processes\thanks{The work is supported by RFBR, project No~20-01-00575 A.}}
%
%
\author{Galina Zverkina\footnote{
V. A. Trapeznikov Institute of Control Sciences of Russian Academy of Sciences,
65 Profsoyuznaya street, Moscow 117997, Russia}}

\maketitle 
\begin{abstract}
Generalization of the Lorden's inequality is an excellent tool for obtaining strong upper bounds for the convergence rate for various complicated stochastic models.
This paper demonstrates a method for obtaining such bounds for some generalization of the Markov modulated Poisson process (MMPP).
The proposed method can be applied in the reliability and  queuing theory.

\noindent {\bf keywords:}
{Markov modulated Poisson process \and Lorden's inequality \and Convergence rate \and Strong upper bounds \and Coupling method \and Successful coupling.}
\end{abstract}
\section{Introduction}
As well-known, a Markov modulated Poisson Process (MMPP) is a Poisson process whose rate varies according to some Markov process.
There is some Markov process  $X_t$, and the behaviour of the other process $Y_t$ at the time $t$ depends on the state of $X_t$ at this time.
Because Poisson process are considered, the processes $X_t$ and $Y_t$ are Markov chains in continuous time, obviously.

Thus, the ``rate'' or {\it intensity} and transition function of the process $Y_t$ depends on the state of $X_t$.
The paired process $(X_t,Y_t)$ is Markov, and it has the constant intensities between the changes of the state of $(X_t,Y_t)$.
This is a multi-dimensional Markov chain in continuous time.

Obviously, the processes $X_t$ and $Y_t$ are one-dimensional, but they may be multi-dimensional.

 The behavior of a large class of queueing systems, reliability systems, networks may be described in terms of MMPP under the assumption about exponential distribution of all random variables (r.v.'s).

Here, the generalization of MMPP is given; we suppose that random variables in the model can have arbitrary distributions.
Moreover, we suppose, that the processes $X_t$ and $Y_t$ can be multi-dimensional, and the intensities of everyone components of the paired process $(X_t,Y_t)$ can be depended on everyone.
Such a process describes a wide class of models of queueing systems, reliability systems, and networks.

So, we consider the process $X_t=(X_t^{(1)}, X_t^{(2)}, X_t^{(3)}, \ldots X_t^{(n)})$, where the processes $X_t^{(i)}$ are one-dimensional Markov chains in continuous time.
For simplicity, we suppose that all these Markov chains in continuous time have only one state of embedded Markov chain (i.e. some renewal processes).
In other words, the processes are flows of events (not Poisson!).

In general, this multidimensional process is non-regenerative.

It is well known that in a queuing theory and related disciplines it is very important to establish the ergodicity of the systems, as well as their stationary distribution.
Also it is important to know what is the rate of convergence of the distribution of the system to the limit distribution.
This is important because when we use some systems in practice, we need to know when the stationary distribution can be used in calculations and for estimations instead of the time-varying distribution of the system.

Our aim is to prove the ergodicity of the process $X_t$ under some assumptions about the distributions of random times between the events of all processes $X_t^{(i)}$.

\section{Some preliminary considerations}
\begin{definition}
[Regenerative Process]
A random process is called regenerative if there exists an increasing sequence $\{t_i\}_{i=0,1,2,\ldots}$, such, that the random elements $\Theta_i\stackrel{{\rm def}}{=\!\!\!=}  \{X_t, \, t\in [t_{i-1};t_i]\}$ are i.i.d. $\forall \;i=1,2,\ldots$.

Times $t_i$ are named {\it regeneration times}.

Denote $\tau_i\stackrel{{\rm def}}{=\!\!\!=}  t_{i+1}-t_{i}$, and let $\mathscr {P} _t$ be a distribution of regenerative process at the time $t$.
\hfill\ensuremath{\triangleright}
\end{definition}

It is well-known that:

1. If $\mathbf{E} \,\tau_i<C<\infty$, then the regenerative process is ergodic, i.e. there exists the probability distribution $\mathscr {P} $, such that $\mathscr {P} _t\Longrightarrow \mathscr {P} $.

2. If $\mathbf{E} \,(\tau_i)^k<\infty$, then $\exists \; K(\mathscr {P} _0):$
\begin{equation}\label{poly}
\|\mathscr {P} _t-\mathscr {P} \|_{TV}\le \frac{K((k-1),\mathscr {P} _0)}{t^{k-1}}.
\end{equation}

3. If $\mathbf{E} \,\exp(\alpha\tau_i)<\infty$, then $\forall \;\beta<\alpha\; \exists \; K(\mathscr {P} _0,\beta):$
$$
\|\mathscr {P} _t-\mathscr {P} \|_{TV}\le K(\mathscr {P} _0,\beta) \exp{(-\beta t)}.
$$
These results are the classic results, but they make it impossible to estimate the value $K$ -- see, e.g., \cite{Asm,Thor,GK,Af} et all.

The convergence rate can be estimated explicitly in a fairly limited number of cases.

These are situations when, for example, in the QS the incoming flow is Poisson, and the service times are distributed exponentially.
The same approach is applied to the study of reliability for systems consisting of restorable elements: the rate of convergence of the distribution of such a reliability system can be estimated when the work and repair times are exponential (see, e.g., \cite{GBS}).
It is always assumed that all switching (transitions between operating and repair modes) occur instantly.

The goal of the present paper is to  estimate the convergence rate of the QS distribution, consisting of {\it dependent} servers, and the connection of which may be delayed.

Previously, such systems were studied using the special Lyapunov function -- see, e.g., \cite{ver1,ver2}.

Note, that the regenerative process have an embedded renewal process, and the bounds of the convergence rate of this renewal process in total variation metrics are also the bounds of the convergence rate of studied regenerative process.
This fact is the basis for obtaining strong upper bounds for the convergence rate of regenerative processes.
Also, this general method for obtaining an upper bounds for the constant $K$ was invented in \cite{zv2}.
However, special methods are desirable for specific QS.

All known upper bounds for the constants for the convergence rate are based on the coupling method, invented in \cite{zv6}.

\subsection{Some information about the coupling method}
Consider two \emph{independent Markov processes} with different initial states $X_0$ and $\widehat{X}_0$ and with the same transition function.
Denote these processes by $X_t $ and $\widehat{X}_t$ correspondingly.
Suppose we can find the time $\tau$ where they are coincided.
The time $\tau$ is called \emph{coupling epoch} and it depends on $X_0$ and $\widehat{X}_0$.
After the time $\tau(X_0 ,\widehat{X}_0)$, the distributions of the processes $X_0 $ and $\widehat{X}_0$ are coincided by Markov property.
Thus, for all $t\ge \tau(X_0 ,\widehat{X}_0)$, and for all set $\mathscr {A}\in \sigma(\mathscr {X})$, $\mathbf{P}\{X_t \in \mathscr {A}\}=\mathbf{P}\{\widehat{X}_t\in \mathscr {A}\}$.
It implies the basic coupling inequality:
\begin{eqnarray*}
|\mathbf{P} \{X_t \in \mathscr {A}\} - \mathbf{P} \{{\widehat{X}}_t\in \mathscr {A}\}|
=|\mathbf{P} \{X_t \in \mathscr {A}\;\&\; \tau>t\} - \mathbf{P} \{{\widehat{X}}_t\in \mathscr {A}\;\&\; \tau>t\} + \nonumber
\\ \\
+ |\mathbf{P} \{X_t \in \mathscr {A}\;\&\; \tau \le t\} - \mathbf{P} \{{\widehat{X}}_t\in \mathscr {A}\;\&\; \tau \le t\}|= \hspace{2.5cm} 
\\ \\ \nonumber
=|\mathbf{P} \{X_t \in \mathscr {A}\;\&\; \tau>t\} - \mathbf{P} \{{\widehat{X}}_t\in \mathscr {A}\;\&\; \tau>t\} |\le \mathbf{P} \{\tau>t\}. \nonumber
\end{eqnarray*}
Then, if it is possible to find the increasing positive function $\varphi(\tau)$ such that $\mathbf{E}\,\varphi(\tau(X_0 ,\widehat{X}_0))<\infty$, then by Markov inequality,
$$
\mathbf{P} \{\tau (X_0 ,\widehat{X}_0)\ge t\}=\linebreak \mathbf{P} \{\varphi(\tau(X_0 ,\widehat{X}_0)) \ge \varphi(t)\}\le \displaystyle \frac{\mathbf{E} \,\varphi(\tau(X_0 ,\widehat{X}_0))}{\varphi(t)}.
$$
From the last inequality the bounds for convergence of the distribution $\mathscr {P}_t$ can be obtained.
Let's explain it.

If the process $\widehat{X}$ starts from the stationary distribution $\mathscr  P$ of $X_t$, i.e. at any time, the distribution of $\widehat{X}_t$ is the same as the one of $\widehat{X}_0$, then
$$
|\mathbf{P} \{X_t \in \mathscr {A}\} - \mathscr {P} (A)\}|\le \frac{\displaystyle\int\limits _\mathscr  X\varphi(\tau(X_0 ,\widehat{X}_0))\mathscr  P(\,\mathrm{d} \, \widehat{X}_0)}{\varphi(t)},
$$
where $\mathscr {X} $ is the state space of $X_t$

This schema can be used for discrete Markov chain and for Markov chain in continuous time.
But for the process $X_t$ in continuous time, the ``direct'' coupling method is impossible, because for different values $X_0\neq \widehat X_0$, $\mathbf{P}\{\tau(X_0,\widehat X_0)<\infty\}=0$.
Thus, the modification of coupling method, or \emph{successful coupling} will be used.

\subsection{Successful coupling (see \cite{zv5}).}\label{sc} 
Let $X_t$ and $\widehat{X}_t$ be two independent Markov processes with the same transition function, but with different initial states at time $t=0$.

Suppose that (\emph{dependent}) processes $Y_t$ and $\widehat{Y}_t$ are constructed on some probability space, in such a way that:
\\

\noindent \textbf{1.} $Y_t\stackrel{\mathscr {D}}{=}X_t$ and $\widehat{Y}_t\stackrel{\mathscr {D}}{=}\widehat{X}_t$ for all \emph{non-random} $t$;\\ \\
\textbf{2.} $\mathbf{P} \{\tau(X_0,\widehat{X}_0)<\infty\}=1$, where $\tau(X_0,\widehat{X}_0)=\tau(Y_0,\widehat{Y}_0)=\inf\{t>0:\; Y_t=\widehat{Y}_t\}$.\\

This pair of processes $Y_t$ and $\widehat{Y}_t$ is called \emph{successful coupling} for the processes $X_t$ and $\widehat{Y}_t$, and $\tau (X_0,\widehat{X}_0)$ is called \emph{ coupling epoch}.

For successful coupling, the basic coupling inequality can be applied as:
\begin{eqnarray}
|\mathbf{P} \{X_t\in \mathscr {A}\} - \mathbf{P} \{\widehat{X}_t\in \mathscr {A}\}|= |\mathbf{P} \{Y_t\in \mathscr {A}\} - \mathbf{P} \{\widehat{Y}_t\in \mathscr {A}\}|=\hspace{3cm} \nonumber
\\ \nonumber\\\nonumber
=|\mathbf{P} \{Y_t\in \mathscr {A}\;\&\; \tau>t\} - \mathbf{P} \{\widehat{Y}_t\in \mathscr {A}\;\&\; \tau>t\} + \nonumber \hspace{4cm}
\\\nonumber \\
+ |\mathbf{P} \{Y_t\in \mathscr {A}\;\&\; \tau\le t\} - \mathbf{P} \{\widehat{Y}_t\in \mathscr {A}\;\&\; \tau\le t\}|= \hspace{3.5cm} \label{zv12} 
\\ \nonumber \\ \nonumber
=|\mathbf{P} \{Y_t\in \mathscr {A}\;\&\; \tau> t\} - \mathbf{P} \{\widehat{Y}_t\in \mathscr {A}\;\&\; \tau> t\} |\le \mathbf{P} \{\tau> t\}
\end{eqnarray}
for any set $\mathscr {A}\in\sigma(\mathscr {X})$.
Here, identical distribution of pairs $Y_t\stackrel{\mathscr {D}}{=} X_t$ and $\widehat{Y}_t\stackrel{\mathscr {D}}{=} \widehat{X}_t$ means
that only marginal distributions coincide in any time, but not finite-dimensional distributions of these processes.

Now, our goal is a construction of the successful coupling and an estimation of polynomial moments of a random variable $\tau(X_0,\widehat{X}_0)$.
For this construction the Basic Coupling Lemma is needed.

\subsection{Basic Coupling Lemma (see, e.g., \cite{zv13,verbut,GZ1}).}
Here the simplest formulation of the Basic Coupling Lemma is given.

\begin{lemma}\label{osn} 
If the random variable $\vartheta_1$ and $\vartheta_2$ have c.d.f. $\Phi_1(s)$ and $\Phi_2(s)$ cor\-res\-pon\-din\-g\-ly, and their common part $\varkappa\stackrel{\mbox{\rm def}}{=\!\!\!=} \int\limits_\mathbf{R} \min\{\Phi_1'(s), \Phi_2'(s)\}\,\mathrm{d}\, s>0$, then the random variables $\widehat{\vartheta}_1$ and $\widehat{\vartheta}_2$ can be constructed (on some probability space) such, that

1. $\widehat{\vartheta}_1\stackrel{\mathscr {D}}{=} {\vartheta}_1$, $\widehat{\vartheta}_2\stackrel{\mathscr {D}}{=} {\vartheta}_2$;\qquad
2. $\mathbf{P} \{\widehat{\vartheta}_1= \widehat{\vartheta}_2\}=\varkappa$. \hfill \ensuremath{\triangleright}
\end{lemma}

The statement of Lemma \ref{osn} has been extended to any finite number of random variables.
\begin{lemma}\label{lem4}  
Let $\vartheta_1$, $\vartheta_2$, $\ldots$, $\vartheta_n$ be the random variable with probability densities $\varphi_1(s)$, $\varphi_2(s)$, $\ldots$, $\varphi_n(s)$ correspondingly, and $\varkappa \stackrel{\mbox{\rm def}}{=\!\!\!=}\int\limits_\mathbf{R}\min\limits_{i=1,\ldots,n}\{\varphi_i(s)\}\, \mathrm{d}\, s>0$.
Then on some probabilistic space it is possible to construct the random variables $\widehat{\vartheta}_1(s)$, $\widehat{\vartheta}_2(s)$, $\ldots$, $\widehat{\vartheta}_n(s)$ such that

$1.\;\;\widehat{\vartheta}_i\stackrel{\mathscr {D}}{=} {\vartheta}_i$, $i=1,2,\ldots n;\qquad
2. \;\;\mathbf{P} \{\widehat{\vartheta}_1= \widehat{\vartheta}_2=\ldots=\widehat{\vartheta}_n\}=\varkappa$. \hfill\ensuremath{\triangleright}
\end{lemma}
\begin{proof} Let $\varkappa<1$ (the proof for the case $\varkappa=1$ is very simple). Consider a probability space $(\Omega,\sigma(\Omega),\mathbf{P})$, where $\Omega=[0;1)^{n+1}=[0;1)_1\times[0;1)_2\times\cdots\times[0;1)_{n+1}$, $\sigma(\Omega)$ is its Borel $\sigma$-algebra, and $\mathbf{P}$ is Lebesgue measure on $\Omega$.

Let $\mathscr {U}_i$ be the random variable with continuous uniform distribution on $[0;1)_i$, \linebreak$i=1,\ldots,(n+1)$.

Let $\varphi(s)\stackrel{\mbox{\rm def}}{=\!\!\!=}\min\limits_{i=1,\ldots,n}\varphi_i(s)$, and
$$\Psi(s)\stackrel{\mbox{\rm def}}{=\!\!\!=}\frac{1}{\varkappa}\int\limits_{-\infty}^s \varphi (u)\, \mathrm{d}\, u, \qquad \Psi_i(s)\stackrel{\mbox{\rm def}}{=\!\!\!=}\frac{1}{1-\varkappa}\int\limits_{-\infty}^s (\varphi_i (u)-\varphi(u))\, \mathrm{d}\, u.
$$
$\Psi$ and $\Psi_i$ are the distribution functions.

Put $\widehat{\vartheta}_i\stackrel{\mbox{\rm def}}{=\!\!\!=} \Psi_i^{-1}(\mathscr {U}_i)\times \textbf{1}(\mathscr {U}_{n+1}>\varkappa)+ \Psi^{-1}(\mathscr {U}_i)\times \textbf{1}(\mathscr {U}_{n+1}\le\kappa)$, $i=1,\ldots,n$.
It is easy to see that the random variables $\widehat{\vartheta}_i$ satisfy the conditions 1 and 2 of Lemma \ref{lem4}. \hfill \ensuremath{\triangleright}
\end{proof}

The Lorden's inequality has been used in \cite{zv1,zv2,zvpol} for coupling method application to find upper bounds of convergence rate for regenerative processes.

Now, we need to use some {\it generalization} of the Lorden's inequality.

\section{Lorden's inequality}\label{sec2} 

Consider the renewal process $N_t\stackrel{{\rm def}}{=\!\!\!=} \displaystyle \sum\limits  _{i=1}^\infty \mathbf{1} \left\{ \sum\limits  _{k=1}^i \xi_k\le t\right\} $, where
$\left\{\xi_1, \xi_2, ...\right\} $ is the of i.i.d. positive random variables.
$N_t$ is a counting process that changes its value at the times $t_k=S_k\stackrel{{\rm def}}{=\!\!\!=} \displaystyle \sum\limits  _{j=1}^k \xi_j$.
The times $t_k$ are the renewal times.

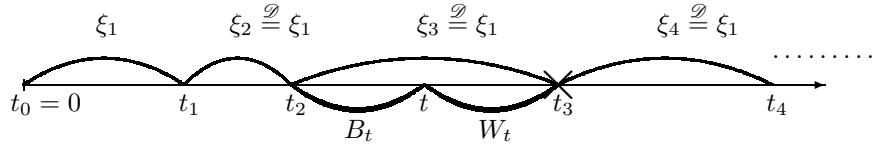
\begin{figure}[ht]
\centering
\begin{picture}(300,45)
\put(0,20){\vector(1,0){300}}
\put(0,18){\line(0,1){4}}
\thicklines
{\qbezier(0,20)(30,40)(60,20)}
{\qbezier(60,20)(80,40)(100,20)}
\qbezier(100,20)(150,40)(200,20)
\qbezier(200,20)(240,40)(280,20)
\put(280,30){\ldots \ldots \ldots}
\put(58,10){$t_1$}
\put(98,10){$t_2$}
\put(198,10){$t_3$}
\put(278,10){$t_4$}
\put(27,40){$\xi_1$}
\put(77,40){$\xi_2\stackrel{\mathscr {D}}{=} \xi_1$}
\put(147,40){$\xi_3\stackrel{\mathscr {D}}{=} \xi_1$}
\put(237,40){$\xi_4\stackrel{\mathscr {D}}{=} \xi_1$}
\put(-5,10){$t_0=0$}
\put(148,10){$t$}
\put(195,15){\line(1,1){10}}
\put(195,25){\line(1,-1){10}}
\qbezier(150,20)(175,0)(200,20)
\qbezier(150,20)(175,1)(200,20)
\qbezier(150,20)(175,2)(200,20)
\put(170,0){$W_t$}
\qbezier(100,20)(125,0)(150,20)
\qbezier(100,20)(125,1)(150,20)
\qbezier(100,20)(125,2)(150,20)
\put(120,0){$B_t$}
%
\end{picture}
\caption{$B_t$ is a backward renewal time, and $W_t$ is a forward renewal time at the fixed time $t$; $B_t=t-\displaystyle \sum\limits  _{1}^{N_{t}}\xi_i$.}
\label{fig1}
\end{figure}
In Fig.\ref{fig1}, we see the {\it backward renewal time (or overshoot)} $B_t$, and the {\it forward renewal time (or undershot)} $W_t$ at the {\bf fixed} time $t$ (See Fig.\ref{fig1}):
$$B_t\stackrel{{\rm def}}{=\!\!\!=}  t -S_{N_t}; \qquad W_t\stackrel{{\rm def}}{=\!\!\!=}  S_{N_t+1} -t.$$

\begin{theorem}[Lorden, G. (1970) \cite{Lorden}; see, e.g. \cite{Chang}]\label{thm} 
Lorden's inequality states that the expectation of this overshoot is bounded as
$$
\mathbf{E} \, B_t \le \frac{\mathbf{E} \, \xi_1^2}{\mathbf{E} \, \xi_1}. \eqno{\triangleright}
$$
\end{theorem}

This inequality is the tool for upper bounds for convergence rate in total variation metrics.

However, we need to use some generalization of the Lorden's inequality.

\subsection{Generalized Lorden's inequality}
\begin{definition}[Quasi-renewal process]
Consider the counting process \linebreak$N_t\stackrel{{\rm def}}{=\!\!\!=} \displaystyle \sum\limits  _{i=1}^\infty \mathbf{1} \left\{ \sum\limits  _{k=1}^i \xi_k\le t\right\} $, where
$\left\{\xi_1, \xi_2, ...\right\} $ are positive random variables is named quasi regenerative if the r.v.'s $\xi_i$ are not  i.i.d.

As for a classic renewal process, $N_t$ is a counting process that changes its value at the times $t_k=S_k\stackrel{{\rm def}}{=\!\!\!=} \displaystyle \sum\limits  _{j=1}^k \xi_j$.
The times $t_k$ are the renewal times. \hfill\ensuremath{\triangleright}
\end{definition}
The generalized Lorden's inequality estimates the expectation of the backward renewal time of quasi-renewal process -- in some condition, naturally.
\subsection{Some preliminary considerations about generalized intensities.}
The random variables can be defined by distribution function, by its density, and by its {\it intensity}.
Obviously, the intensity is a means for study the absolutely continuous distributions, and for d.f. $F(x)$ the intensity is equal $\lambda(x)=\displaystyle \frac{F'(x)}{1-F(x)}$.

This formula is correct:

\begin{equation}\label{inten}
F(x)=1-\exp\left(\int\limits _0^x -\lambda(s)\,\mathrm{d} \, s\right).
\end{equation}

But we are interested in mixed random variables, i.e. their distribution functions can have jumps.

If $F(a-0)\neq F(a+0)$, the {\bf we put} $\lambda(a)\stackrel{{\rm def}}{=\!\!\!=} -\ln \big(F(a+0)-F(a-0) \big)\delta(0)$, where $\delta(\cdot)$ is a ``classic'' $\delta$-function.
So, {\bf we put}
$$
f(s)=\left\{
\begin{array}{ll}
F'(s), & \mbox{if }\exists\; F'(s);\\ \\
0, & \mbox{otherwise}.
\end{array}
\right.
$$
Finally, (generalized) intensity is
$$\lambda(s)\stackrel{{\rm def}}{=\!\!\!=}  \displaystyle \frac{f(s)}{1-F(s)}-\sum\limits  _{i}\delta(s-a_i)\ln\big(F(a_i+0)-F(a_i-0) \big),$$ where $\{a_i\}$ is a set of the discontinuous points of d.f. $F(s)$.

It is easy to see, that the formula (\ref{inten}) remains true for generalized intensity.

Let $Int_{\xi}(s)$ be an intensity of r.v. $\xi$.

We skip very easy proof of the next Lemma \ref{min}.
\begin{lemma}\label{min} 
$Int_{\min\{\xi;\eta\}}=Int_{\xi}+Int_{\eta}$. \hfill\ensuremath{\triangleright}
\end{lemma}

\subsection {Notations and assumptions.} Consider the sequence $\left\{\xi_1, \xi_2, ...\right\} $ of random variables.
\\\\
{\bf Assumptions}
\begin{enumerate}
{\it
\item $\xi_j=\min\{\zeta_j;\eta_j\}$, were $\{\zeta_j\}$ are i.i.d. r.v.'s defined by (generalized) intensities $\varphi_i(s)\equiv \varphi(s)$, and $\zeta_i \perp\!\!\!\perp \eta_j$ for all $i$, $j$; $\eta_j$ are defined by (generalized) intensities $\mu_j(s)$;
\item There exists some (generalized) measurable function $Q(s)$ such that for all $s\ge 0$, $\varphi(s)+\mu_j(s)=\lambda_i(s) \le Q(s)$;
\item $\displaystyle \int\limits _0^\infty \varphi(s)\,\mathrm{d} \, s = \infty$, and $\displaystyle \int\limits _0^\infty \left( x^{k-1} \exp\left(-\int\limits _0^x \varphi(s)\,\mathrm{d} \, s\right)\right)\,\mathrm{d} \, x<\infty $ for some $k\ge 2$;
\item $Q(s)$ is bounded in some neighborhood of zero;
\item $\varphi(s)>0$ a.s. for $s>T \ge 0$.} \hfill\ensuremath{\triangleright}
\end{enumerate}

\begin{definition}
If conditions 1--4 are satisfied, then the counting process
\begin{equation}\label{def}  
N_t\stackrel{{\rm def}}{=\!\!\!=} \sum\limits  _{i=1}^\infty \mathbf{1} \left\{\sum\limits  _{k=1}^i \xi_k\le t\right\}
\end{equation}
is named {\it generalized renewal process}.
\hfill\ensuremath{\triangleright}
\end{definition}

For this quasi-renewal process the r.v.'s $B_t\stackrel{{\rm def}}{=\!\!\!=}   t -S_{N_t}$ and $ W_t\stackrel{{\rm def}}{=\!\!\!=}  S_{N_t+1} -t$ are backward renewal time and forward renewal time accordingly.
\begin{remark}\label{Wt}  
The r.v. $W_t$ also acn be represented as \linebreak $W_t=\min\{\zeta(B_t), \eta(B_t)\}$, where $\zeta(B_t)$ has the distribution function \linebreak $\mathbf{P} \{\zeta(B_t)\le s\} =\displaystyle \frac{\mathbf{P} \{\zeta_{N_t+1}-B_t\le s\}}{1-\mathbf{P} \{\zeta_{N_t+1}> B_t\}} $, and $\eta(B_t)$ has the distribution function $\mathbf{P} \{\eta(B_t)\le s\} =\displaystyle \frac{\mathbf{P} \{\eta_{N_t+1}-B_t\le s\}}{1-\mathbf{P} \{\eta_{N_t+1}> B_t\}}$. \hfill\ensuremath{\triangleright}
\end{remark}

\begin{remark}\label{r00} 
The assumption 1 holds:
the r.v.'s $\xi_i$ and $\xi_j$ are dependent, and this dependency is ``weak'' dependency in some sens.

Also the intensity of renewal time of considered quasi-renewal process is a sum of two independent processes: one of them is a classic renewal process (with renewal times $\zeta_i$) and quasi-renewal process with (dependent) renewal times $\eta_i$.

Such a process describes some process that has some minimal intensity and additional increase in intensity due to the influence of some external factors.
\hfill\ensuremath{\triangleright}
\end{remark}

\begin{remark}\label{r0}
The assumptions 3 and 4 hold:
$\mathbf{E} \,\xi_i>0,\qquad \mbox{Var}\,\xi_i^2>0.$
\hfill\ensuremath{\triangleright}
\end{remark}

\begin{remark}\label{r1} 
The assumptions 1 and 2 hold: 
$$F_i(t)=\mathbf{P} \{\xi_i\le t\}=1-\displaystyle \exp\left(\int\limits _0^t {-\lambda_i(s)}\,\mathrm{d} \, s\right) \ge 1- \displaystyle \exp\left(\int\limits _0^t {-Q(s)}\,\mathrm{d} \, s\right)$$

$ \Rightarrow \;\;\exists\; \mathbf{E} \,\xi_i^2 < \infty.$
\hfill\ensuremath{\triangleright}
\end{remark}

\begin{remark}\label{r2} 
The assumption 5 reports that the renewal process under study is a delay renewal process, and a delay time does not exceed $T$.
\hfill\ensuremath{\triangleright}
\end{remark}

\begin{theorem}[Generalized Lorden's inequality -- see \cite{kalzv}] \label{thm1} 
If the conditions 1--4 are satisfied, then for the process (\ref{def}) the inequality

\begin{equation}\label{osn1}
\mathbf{E} \,B_t\le \mathbf{E} \, \zeta + \frac{\mathbf{E} \,\zeta^2}{2\mathbf{E} \, \xi}=\Xi(=\Xi(1)),
\end{equation}
is true, where
$ \mathbf{E} \, \zeta= \displaystyle \int\limits _0^\infty x \,\mathrm{d} \, \Phi(x);\quad
\mathbf{E} \, \xi=\displaystyle  \int\limits _0^\infty x \,\mathrm{d} \, G(x),$ and
\\
$
G(x)= 1- \displaystyle\exp \left(-\int\limits _0^x Q(t)\,\mathrm{d} \, t\right) , \mbox{ and } \Phi(x)= 1- \exp \left(-\int\limits _0^x \varphi(t)\,\mathrm{d} \, t\right) .
 $\hfill\ensuremath{\triangleright}
\end{theorem}
\begin{remark}\label{W} 
In the proof of this Theorem, there is an intermediate result:

If $\mathbf{E} \, \zeta^k<\infty$, then for $N\in(0;k-1]$,
$$
\displaystyle \mathbf{E} \,(B_t)^N\le
=\mathbf{E} \,\zeta^N+\frac{\mathbf{E} \, \zeta^{N+1}}{(N+1)\mathbf{E} \, \xi}\stackrel{{\rm def}}{=\!\!\!=}  \Xi(N). \eqno {\triangleright}
$$
\end{remark}


\section{Considered generalized Markov-modulated Poisson process}
Consider the process $X_t\stackrel{{\rm def}}{=\!\!\!=}  (X_t^{(1)}, X_t^{(2)}, X_t^{(3)}, \ldots X_t^{(m)})$, where the processes $X_t^{(i)}$, $i=1,2, \ldots m$ are generalized renewal processes $X_t^{(i)}$, $i=1,2, \ldots m$.

For these processes, the backward renewal times $B_t^{(i)}$ and forward renewal times $W_t^{(i)}$ are for processes $X_t^{(i)}$ accordingly defined in the Remark \ref{Wt}.

At the time $t$, the state of the process $X^{(i)}_t$ is $x^{(i)}_t=t-\displaystyle \sum\limits  _{j=1}^{N_t^{(i)}}\xi_i^{(j)}(=B_t^{(i)}$ -- elapsed renewal time or backward renewal time of process $ X_t^{(i)})$  -- see Fig. \ref{fig1}.

The intensities of the processes $X_t^{(i)}$, $i=1,2, \ldots m$ are $\varphi(B_t^{(i)})+\mu_i(B_t^{(i)})$.

Renewal times of qusi-renewal process $X_t^{(i)}$ are $\xi_i^{(j)}=\min\{\zeta_i^{(j)},\eta_i^{(j)}\}$, where $\zeta_i^{(j)}$ has intensity $\varphi(x^{(i)}_t)$, and $\eta_i^{(j)}$ has intensity $\mu_i(X_t)$.

So, the ``full'' intensity of $X_t^{(i)}$ (which corresponds to $\zeta_i^{(j)}$ in Assumption 1) is $\lambda_i(X_t)=\varphi(x^{(i)}_t)+\mu_i(X_t)$.
``Additional'' intensity (which corresponds to $\eta_i^{(j)}$ in Assumption 1) is $\mu_i(t)=\mu(x^{(1)}_t, x^{(2)}_t, x^{(3)}_t, \ldots x^{(m)}_t)$.

For simplicity, here we replace inequalities $\lambda_i(x)\ge \varphi_i(x)$ and $\lambda_i(x)+\mu_i(x)\le Q_i(x)$ by homogeneous inequalities $\lambda_i(x)\ge \varphi(x)$ and $\lambda_i(x)+\mu_i(x)\le Q(x)$.


The processes $X^{(i)}_t$ are dependent, and the intensity of $X^{(i)}_t$ depends on the state of the processes $X^{(j)}_t$, $j\neq i$.

The process $X_t\stackrel{{\rm def}}{=\!\!\!=}  (X^{(1)}_t, X^{(2)}_t, X^{(3)}_t, \ldots X^{(m)}_t)$ is not regenerative in general.

It can be considered as a kind of variation on the topic of the Markov modulated Poisson process (see, e.g., \cite{mmpp1,mmpp2})

Emphasize, that for simplicity, here we suppose, that the functions $\varphi(s)$ and $Q(s)$ are the same for all processes.

~

\noindent {\bf Additional assumption}
\begin{enumerate}
\item[6.] For all $t>0$, $\displaystyle \int\limits _0^s Q(s)\,\mathrm{d} \, s<\infty$, and there exists non-negative function $q(s)$ such that $\mu_i(s)=\lambda_i(s)-\varphi(s)\ge q(s)$.
\end{enumerate}
\begin{remark}\label{distr} 
 From Assumptions 1 and 6 we have the formulae for distributions of $\zeta$ and $\eta$.
$$\mathbf{P} \{\zeta_i\le s\}=\Phi(s)> 0\mbox{ for }s>T,$$
and
$$\mathbf{P} \{\eta_i\le s\}=\widetilde  G(s)\stackrel{{\rm def}}{=\!\!\!=}  1- \exp \left(-\int\limits _0^s (Q(t)-\varphi(t))\,\mathrm{d} \, t\right);$$
$$\mathbf{P} \{\eta_i> s\}=1-\widetilde  G(s)\ge\exp \left(-\int\limits _0^s Q(t)\,\mathrm{d} \, t\right)=1-G(s).$$
Note, that $\mathbf{P} \{\zeta<\infty\}=1$, and $\mathbf{P} \{\eta<\infty\} \le 1$.

Assumptions 1--6 hold: the probability densities of the r.v.'s $\xi_j^{(i)}$ can be bounded as:
\begin{equation}\label{ost}
\begin{array}{l}
F_j(x)=\mathbf{P} \{\xi_j^{(i)}\le x\}=1-\displaystyle \exp\left( -\int\limits _0^x (\lambda_j(s)+\mu_j(s))\,\mathrm{d} \, s\right)\le
\\
\hspace{6cm}\le \displaystyle  1-\exp\left( -\int\limits _0^x \varphi(s)\,\mathrm{d} \, s\right);
\\
f_j(x)=\displaystyle  \frac{\,\mathrm{d} \, F_j(x)}{\,\mathrm{d} \, x}\ge \exp\left( -\int\limits _0^x Q(s)\,\mathrm{d} \, s\right)\varphi(s)>0\mbox{ a.s. for $t>T$}.
\end{array}
\end{equation}

For forward renewal times $W_t^{(i)}(a)$ of the processes $X_t^{(i)}$ given $B_t^{(i)}>a$ have the probability densities:
\begin{equation}\label{oc1}
f_t^{(i)}(x,a)\ge \varphi(a+x)\exp\left( -\int\limits _a^{x+a} Q(s)\,\mathrm{d} \, s\right)\stackrel{{\rm def}}{=\!\!\!=} f(a,x)>0\mbox{ a.s. for $t>T$},
\end{equation}
see the Assumption 5.
\hfill\ensuremath{\triangleright}
\end{remark}


$\mathbf{P} \{ \zeta_t^{(i)}>\eta_t^{(i)}| B_t^{(i)}<\Theta\}\ge A(\Theta)$.

\subsection{Ergodicity of the multi-dimensional process  $X_t$}\label{ergo}
Here we give {\it only} the schema for the proof of the ergodicity of the multi-dimensional process $X_t$.

Suppose, $X_0=(0,0,0, \ldots, 0)$.


Put $\theta_0=0$.
This is markov moment of $X_t$.

At these times, the backward renewal time $B_{\theta_1}^{(i)}$  satisfies the inequality (\ref{osn1}).
Put constant $\Theta>\Xi$.
By Markov inequality, \\
$$\displaystyle  \mathbf{P} \{B^{(i)}_{{\theta_1} }\ge\Theta\}\le \frac{\Xi}{\Theta},\mbox{ and  }\displaystyle \mathbf{P} \{B ^{(i)}_{\theta_1} <\Theta\}\ge 1- \frac{\Xi}{\Theta}=\pi_0(\Theta).
$$

These bounds are uniform for all the processes $X_t^{(i)}$, they don't depend on the number $i$ of the process $X_t^{(i)}$.

So, at the time ${\theta_\ell} $ for all the processes $X_t^{(i)}$, $B^{(i)}_{\theta_1} <\Theta$, 
with probability greater than $p_0\stackrel{{\rm def}}{=\!\!\!=}  (\pi_0(\Theta))^m$.

At this time, the forward renewal times $W_{\theta_1}^{(i)}$ of the processes $X_t^{(i)}$ are the residual times of $\zeta^{(i)}$ given the elapsed time is less then $\Theta$ for all $\zeta^{(i)}$.

Now, we can use the generalization of the Basic Coupling Lemma, and we can create (in some probability space) the prolongation of the processes $X_t^{(i)}$ by such a way, that the next renewal times of both these processes coincide with probability greater then
$$
\pi_1\stackrel{{\rm def}}{=\!\!\!=}  \inf\limits _{a_j\in[0;\Theta], j=1,2,\ldots, m}\int\limits _0^\infty \min\{f(x,a_j)\}\,\mathrm{d} \, x>0\quad - \mbox{ see (\ref{oc1}).}
$$

If at the time $\theta_1=\theta_0+\max\{\xi_1^{(j)}\}=\theta_0+\max\{W_{\theta_0}^{(j)}\}$ all the processes $X_{\theta_0}^{(i)}=0$, we say that $\theta_1$ is the first regeneration point of the constructed version for $X_t$.

In this case $X_{\theta_1}=X_{\theta_0}$, and we repeat the pricedure descripted above.
Otherwise consider $W_{\theta_1}^{(i)}$.
Again with probability greater then $\pi_1$ at the time $\theta_2\stackrel{{\rm def}}{=\!\!\!=} \theta_1+\max\limits_i \{W_{\theta_1}^{(i)}\}$ constructed versions of the processes $X_t^{(i)}$ coincide.

Note that $\mathbf E\, (\theta_2-\theta_1)\le \Xi$ (Theorem \ref{osn}).

Anew we use the generalization  of the Basic Coupling Lemma, and we can create (in some probability space) the prolongation of the processes $X_t^{(i)}$ by such a way, that the next renewal times of both these processes coincide with probability greater then $\pi_1$, and so on.

This operation is repeated at every times ${\theta_\ell} \stackrel{{\rm def}}{=\!\!\!=} \theta_{\ell-1}+\max\limits_i \{W_{\theta_{\ell-1}}^{(i)}\}$.

By this way, on some probability space we construct the new process \linebreak$\widetilde  X_t\stackrel{{\rm def}}{=\!\!\!=}  (\widetilde  X_t^{(1)}, \widetilde  X_t^{(2)}, \widetilde  X_t^{(3)}, \ldots ,\widetilde  X_t^{(m)})$, such that the marginal distributions of $\widetilde  X_t$ and $X_t$ are equal, and after any time ${\theta_\ell} $ the process $\widetilde  X_t$  hits to the state \linebreak$(0,0,0,\ldots,0)$ with probability greater then $p_0\pi_1$.

The time between the hits to this zero-state are distributed as geometrical sum of constants $M$ and the finishing residual time of the last period.
So, it has the finite expectation, and the process $\widetilde  X_t$ is regenerative.
Therefore, it has a limit probability distribution.

In this situation, it is natural to call the process $X_t$ as ``quasi-regenerative process''.

{\it In fact, here we used the scheme of successful coupling, but in a slightly different way.}
\begin{remark}\label{limit} 
The limit distribution of the process $\widetilde  X_t$ can be estimated by the same schema as for a ``classic'' renewal process (see \cite{Smith}):
$$
\mathbf{P} \{\widetilde  X_t^{(i)}>s\}\le \Psi(s)\stackrel{{\rm def}}{=\!\!\!=}  \displaystyle  \frac{\displaystyle \int\limits _0^s (1-\Phi(u))\,\mathrm{d} \, u} {\displaystyle \int\limits _0^\infty (1-G(u))\,\mathrm{d} \, u}. \eqno{\triangleright}
$$
\end{remark}

\subsection{About polynomial convergence of the process $ X_t$ }
Now, consider two independent version of multi-dimensional process $X_t$: $X_t$ and $\widehat  X_t=({\widehat  X_t^{(1)}}, {\widehat  X_t^{(2)}}, {\widehat  X_t^{(3)}}, \ldots , {\widehat  X_t^{(m)}})$.

For simplicity, $x_0=0$, $y_0=0$.
The second process has an arbitrary initial state: ${\widehat  x^{(i)}_0}=a_i$.

So, the residual times ${\widehat  \xi_{1}^{(i)}(a_i)}$ of the processes $\widehat X^{(i)}_t$  have the densities greater then $\displaystyle \varphi(a_i+x)\exp\left( -\int\limits _{a_i}^{x+a_i} Q(s)\,\mathrm{d} \, s\right)$  (see (\ref{den1}).

So, if there exists finite $\displaystyle  \int\limits _0^\infty x^k\exp\left( -\int\limits _0^{x} Q(s)\,\mathrm{d} \, s\right)\,\mathrm{d} \, x<\infty$, then there exist finite expectations of  $\mathbf{E} \,(\widehat  \xi^{(i)}_{1}(a_i))^{k-1}$.

After the time $\theta_0=S_0(a_1, a_2, \ldots, a_m)\stackrel{{\rm def}}{=\!\!\!=}  \displaystyle \max\limits_{i=1,2,3,\ldots,m}\{\widehat \xi_{1}^{(i)}(a_i)\}\le \displaystyle \sum\limits  _{i=1}^m \widehat \xi_{N_t+1}^{(i)}(a_i)$, we can use the generalization of Lorden's inequality.

Denote $\theta_\ell \stackrel{{\rm def}}{=\!\!\!=} \theta_{\ell-1}+\max\limits_{i=1,2,3,\ldots,m}\{ W_{\theta_{\ell-1}}^{(i)}, \widehat W_{\theta_{\ell-1}}^{(i)}\}$.


At the time $\theta_0$, consider probability
$$\Pi\stackrel{{\rm def}}{=\!\!\!=} \mathbf{P} \{B_{\theta_\ell} ^{(i)}<\Theta,\;\widehat  B_{\theta_\ell} ^{(i)}<\Theta, \; i=1,2,3,\ldots ,m\} \;\;-
$$
here the letters with caps refer to the process $\widehat  X_t$.

Similarly to the calculations in the Section \ref{ergo} we have $\Pi\ge (p_0)^2$.

Now the generalized Basic Coupling Lemma gives the probability of coincidence of the prolonged versions of both processes after any time ${\theta_\ell} $ -- this probability  $\ge \pi_1$.

Thus, after any time ${\theta_0} $ the construction of successful coupling leads to the moment of coincidence of both processes in zero-state $(0,0,0,\ldots,0)$ with probability greater then $(p_0)^2\pi_1$ -- this is a {\it coupling epoch} for the processes $X_t$ and $\widehat  X_t$.
Otherwise, we use the generalized Basic Coupling Lemma at the time $\theta_1$, and so on.

So, after any time $\theta_\ell$ with probability $\pi\ge (p_0)^2\pi_1$ (see Lemma \ref{lem4}) we can continue the versions of all processes $X_t$ and $\widehat X_t$ hit to the zero-states $(0,0,0,\ldots 0)$ contemporaneously is a coupling epoch of created on some probability space versions of $X_t$ and $\widehat X_t$ -- see Subsection \ref{sc}.

Let $\mathscr {E}_\ell$ be the event $\{{coupling }\;\; {epoch  } \;\;{happened  }\;\; { right  } \;\;\mbox{after  } \;\;{time } \;\;\theta_\ell\}$.

So, the coupling epoch
$$\tau(a_1, a_2, \ldots, a_m) \le S_0 +\sum\limits_{j=1}^{\nu+1}(\max\limits_i \{W_{\theta_j}^{(i)}\}) +\{\xi^{(1)}_{N_{\theta_\ell}} | \mathscr {E}_\ell \},$$
with probability $\mathbf P \,(\mathscr  E _\ell)$, and  $\mathbf{P} \{\nu-1>n\}\ge (1-\pi)^n$; denote $\widetilde W_j\stackrel{{\rm def}}{=\!\!\!=}  \max\limits_i \{W_{\theta_j}^{(i)}\}$.
$\mathbf E\,(\widetilde W_j)^N\le \Xi(N)$ for $N\le k-1$ -- see Remark \ref{W}.

(Note that $\mathbf E\{ \varsigma|B\}\times \mathbf P\{B\}\le \mathbf E\,\varsigma$ for all r.v. $\varsigma$ and event $B$.)

Thus, using the Jensen's inequality in the form
$$
\left(\sum\limits  _{i=1}^j r_i\right)^ N =j^ N  \left({\sum\limits  _{i=1}^j r_i}\left/{j}\right.\right)^ N \le m^{ N -1}\sum\limits  _{i=1}^j (r_i)^ N ,\quad j>1,\; j\in \mathbf N,
$$
so, for coupling epoch we have for $N\le k-1$:
\begin{eqnarray*}
(\tau(a_1, a_2, \ldots, a_m))^N\le (m+2)^{N-1}\times \hspace{4cm}
\\
\times \left( \displaystyle \sum\limits  _{r=1}^m (\widehat\xi_i^{(r)}(a_i))^N +\left(\sum\limits  _{r=1}^{\nu+1}\widetilde W_r\right)^N +(\{\xi^{(1)}_{N_{\theta_\ell}} | \mathscr {E}_\ell \})^N \right)
\end{eqnarray*}
with probability  $\mathbf P\{\mathscr {E}_\ell\}$, and the last term is a conditional r.v., provided that the coupling epoch is on the end of this r.v.

So,
\begin{eqnarray*}
\mathbf E\,(\tau(a_1, a_2, \ldots, a_m))^N\le (m+2)^{N-1}\times \hspace{6cm}
\\
\times \left(\displaystyle \sum\limits  _{r=1}^m \mathbf E\, (\widehat\xi_i^{(r)}(a_i))^N + \mathbf E\,(\nu+1)^N\times  \Xi(N)^N+ \mathbf E\, (\xi^{(1)}_{N_{\theta_\ell}})^N\right)\stackrel{{\rm def}}{=\!\!\!=}
\\
\stackrel{{\rm def}}{=\!\!\!=}\mathrm{T}(a_1, a_2, \ldots, a_m)_N ,
\end{eqnarray*}
where $\nu$ is geometric r.v. with parameter $\pi$; here $\mathbf E\,(\xi^{(1)}_{N_{\theta_\ell}})^N\le \mathbf E\,(\zeta_1)^N $.

Now, integrating $\mathrm{T}(a_1, a_2, \ldots, a_m)_N$ over a stationary measure bounded in (\ref{limit}) and considering Remark \ref{limit}, we can find an upper bound $K( N )$ in formula (\ref{poly}) for initial zero-state of the process $X_t$.
\begin{remark}
For arbitrary initial state of the process $X_t$, the same schema can be used. \hfill\ensuremath{\triangleright}
\end{remark}
\begin{remark}
The value of $K( N )$ depends on the value $\Theta$.
It can be optimized. \hfill\ensuremath{\triangleright}
\end{remark}
\section{Application for one reliability system}
Consider two elements reliability system.
Let the intensities of the failure and of the repair of the main element depend only on the state of this element.
And let the intensities of the failure and of the repair of the reserve element depend on the full state of this reliability system.

The state of the main element is $(i,x)$, where $i$ is an indicator of work/failure state, and $x$ is the elapsed time in this state.
So, the intensities of failure and of repair are $\lambda(i,x)$.

The state of the reserve element also is is $(j,y)$, where $j$ is an indicator of work/failure state, and $y$ is the elapsed time in this state.
The intensities of failure and of repair of this reserve element are $\widehat  \lambda(j,y)$.

We suppose that these intensities satisfy the assumptions 1--7.

In this case, the full periods (work+repaire) of the main element are i.i.d. random variables.
The full period (work+repaire) of the reserve element are the periods of embedded generalized renewal process.

Thus, we can prove the ergodicity of this reliability system, and calculate the upper bounds for the convergence rate of its distribution.

If we take into account the peculiarities of the distributions of work and repair, and consider not only the times of work beginnings of the main element, these estimates can be improved.

\section{Conclusion}
Here we give only the schema of use of the generalized Lorden's inequality for multi-dimensional process and one application to the reliability theory.
This is only beginning of development of this approach to the finding of strong bounds for various stochastic models.

\section*{Acknowledgments}
The author  is grateful to E.~Yu.~Kalimulina for the great help in preparing this paper.
The work is supported by RFBR, project No~20-01-00575 A.

\end{document}